\documentclass[12pt,reqno]{amsart}

\usepackage{enumerate}
\usepackage[margin=1in]{geometry}
\usepackage{ifpdf}
\usepackage{amsmath}
\usepackage{amsfonts}
\usepackage{amssymb}
\usepackage{amsthm}
\usepackage{mathdots}
\usepackage[hyperfootnotes=false]{hyperref}
\usepackage{setspace}
\usepackage{amsrefs}

\newcommand{\ignore}[1]{}

\newcommand{\abs}[1]{\left\lvert {#1} \right\rvert}
\newcommand{\norm}[1]{\left\lVert {#1} \right\rVert}

\newcommand{\C}{{\mathbb{C}}}
\newcommand{\R}{{\mathbb{R}}}

\newcommand{\bB}{{\mathbb{B}}}

\newcommand{\bP}{{\mathbb{P}}}

\newcommand{\bV}{{\mathbb{V}}}

\newcommand{\sZ}{{\mathcal{Z}}}

\newtheorem{thm}{Theorem}[section]

\newtheorem{prop}[thm]{Proposition}

\newtheorem{cor}[thm]{Corollary}

\newtheorem{lemma}[thm]{Lemma}

\theoremstyle{definition}

\theoremstyle{remark}
\newtheorem{remark}[thm]{Remark}

\author{Ji\v{r}\'i Lebl}
\thanks{The author was in part supported by NSF grant DMS 0900885.}
\address{Department of Mathematics, University of Illinois
at Urbana-Champaign, 
Urbana, IL 61801, USA}
\email{jlebl@math.uiuc.edu}
\date{December 14, 2010}

\ifpdf
\hypersetup{
  pdftitle={Normal forms, Hermitian operators, and CR maps of spheres and hyperquadrics},
  pdfauthor={Jiri Lebl}
}
\fi

\title[Normal forms, Hermitian operators, and CR maps]%
{Normal forms, Hermitian operators, and CR maps of spheres and hyperquadrics}

\begin{document}


\begin{abstract}
We prove and organize some results on the normal forms of Hermitian operators
composed with the Veronese map.  We apply this general
framework to prove two specific theorems in CR geometry.
First, extending a theorem of Faran,
we classify all real-analytic CR maps between
any hyperquadric in $\C^2$ and
any hyperquadric in $\C^3$, resulting in a finite
list of equivalence classes.
Second,
we prove that all
degree-two CR maps of spheres in all dimensions are spherically equivalent to a 
monomial map, thus obtaining an elegant classification
of all degree-two CR sphere maps.
\end{abstract}

\maketitle
\vspace*{-15pt}



\section{Introduction} \label{section:intro}

The purpose of this paper is to use
normal forms of 
Hermitian operators in the study of CR maps.
We strengthen a useful
link between linear algebra and several complex variables and
apply the techniques discussed to
the theory of rational CR maps of spheres and hyperquadrics.
After discussing the classification of such CR maps in terms of their
Hermitian forms, we turn to the
classification of 
Hermitian forms arising from
degree-two rational CR maps of
spheres and hyperquadrics.  For degree-two maps, the classification is the
same as the classification of a pair of Hermitian forms up to
simultaneous
$*$-congruence (matrices $A$ and $B$ are $*$-congruent if there exist a nonsingular
matrix $X$ such that $X^*AX = B$, where $X^*$ is the conjugate transpose),
a classical problem in linear algebra with a solution going back to 1930s,
see the survey \cite{LancasterRodman}.

We will apply the theory developed in this paper to two problems.  First,
extending a result by Faran~\cite{Faran:B2B3}, in
Theorem~\ref{thm:hqclass}, we will finish the classification
of all CR maps of hyperquadrics in dimensions $2$ and $3$.
Second, Ji and Zhang~\cite{JiZhang} classified
degree-two rational CR maps of spheres from source dimension 2.
In Theorem~\ref{thm:deg2class} we 
extend this result to arbitrary source dimension and give an elegant version
of the theorem by proving that all degree-two rational CR maps of spheres
in any dimension are spherically equivalent to a monomial map.
We also study the real-algebraic version of the
CR maps of hyperquadrics
problem in dimensions 2 and 3, arising in the case of
diagonal Hermitian forms.

In CR geometry we often think of a real valued
polynomial $p(z,\bar{z})$ on complex space as the composition of the Veronese
map $\sZ$ with a Hermitian form $B$.  That is, $p(z,\bar{z}) =
\langle B \sZ, \sZ \rangle$.  See \S~\ref{section:hermforms} for more on this
setup.  Writing $B$ as a sum of rank one matrices we will find that $p$
can also be thought of as a composition of a holomorphic map composed with
a diagonal Hermitian form.  When we divide the form 
$\langle B \sZ, \sZ \rangle$ by the defining equation of the source
hyperquadric we get a pair of Hermitian forms.  When this pair is
put into canonical
form we obtain a canonical form of the map up to a natural equivalence
relation.  A crucial point is that for degree-two maps the
two forms are linear in $z$.

We thus make a connection between real polynomials and holomorphic maps to
hyperquadrics. 
A hyperquadric is the zero set of a diagonal Hermitian form, and
is a basic example of a real hypersurface in complex space.  
Usually given in nonhomogeneous coordinates, the hyperquadric
$Q(a,b)$ is defined as
\begin{equation}
Q(a,b) := \{ z \in \C^{a+b} \mid
\abs{z_1}^2 + \cdots + \abs{z_a}^2
- \abs{z_{a+1}}^2 - \cdots - \abs{z_{a+b}}^2 = 1
\} .
\end{equation}
$Q(a,b)$ is a hypersurface only when $a \geq 1$.  Also 
$Q(n,0)$ is the sphere $S^{2n-1}$.

We introduce the natural notion of equivalence for
CR maps of hyperquadrics,
which we will call Q-equivalence.
A map $f \colon Q(a,b) \to Q(c,d)$ is CR if it is continuously differentiable
and satisfies the tangential Cauchy-Riemann equations.
A real-analytic CR map is a restriction of a holomorphic map.
See \cites{BER:book,DAngelo:CR} for more information.
Let $U(a,b)$ be the set of automorphisms of the complex projective
space $\bP^{a+b}$ that preserve
$Q(a,b)$.  In homogeneous coordinates these automorphisms are
invertible matrices,
or linear fractional when working in $\C^{a+b}$.  We will say two CR
maps $f$ and $g$ taking $Q(a,b)$ to $Q(c,d)$ are 
\emph{Q-equivalent} if there are $\tau \in U(c,d)$,
$\chi \in U(a,b)$ such that $f \circ \chi = \tau \circ g$.
In the case of spheres, Q-equivalence is commonly
called \emph{spherical equivalence}.

The problem of classifying CR maps of hyperquadrics has a long history.  For
the sphere case, see for
example~\cites{DLP,DKR,HJX,DAngelo:CR,D:prop1,D:poly,DL:families,DL:complex}
and the references within.  For the general hyperquadric case
see~\cites{BH,HJX,BEH}
and the references within.

Again, let $f \colon Q(a,b) \to Q(c,d)$ be a CR map.
We will
first study the case $a+b=2$ and $2 \leq c+d \leq 3$.
We note that $Q(2,0)$ is equivalent to $Q(1,1)$ by the map
$(z,w) \mapsto (1/z,w/z)$.  Hence we only need to consider $Q(2,0)$
as our source.  Similarly, $Q(1,2)$ is equivalent to $Q(3,0)$.
When $c+d = 2$, then we need only consider maps from ball to ball,
which are Q-equivalent to the identity by a theorem of
Pin\v{c}uk~\cite{Pincuk}.
Faran classified all planar maps in~\cite{Faran:B2B3} and used this
result to classify all ball maps in dimensions 2 and 3.


\begin{thm}[Faran~\cite{Faran:B2B3}] \label{faranb2b3thm}
Let $U \subset Q(2,0)$ be connected and open.
Let $f \colon U \to Q(3,0)$ be nonconstant $C^3$ CR map.
Then $f$ is Q-equivalent to exactly one of the following maps:
\begin{enumerate}[(i)]
\item $(z,w) \mapsto (z,w,0)$.
\item $(z,w) \mapsto (z,zw,w^2)$,
\item \label{faran3map} $(z,w) \mapsto (z^2,\sqrt{2}\,zw,w^2)$,
\item \label{faran4map} $(z,w) \mapsto (z^3,\sqrt{3}\,zw,w^3)$.
\end{enumerate}
\end{thm}


While Faran stated the theorem with the
assumption that
the map is defined on all of $Q(2,0)$, the slight generalization we give
follows from the same proof, or by appealing to the theorem of
Forstneri\v{c}~\cite{Forstneric}.
To completely classify all CR maps of hyperquadrics in dimensions 2 and 3,
it remains to classify maps from $Q(2,0)$ to $Q(2,1)$.



\begin{thm} \label{thm:hqclass}
Let $U \subset Q(2,0)$ be connected and open.
Let $f \colon U \to Q(2,1)$ be a nonconstant real-analytic CR map.
Then $f$ is Q-equivalent to exactly one of the following maps:
\begin{enumerate}[(i)]
\item \label{hqclass1}
$(z,w) \mapsto (z,w,0)$,
\item \label{hqclass2}
$(z,w) \mapsto (z^2,\sqrt{2}\,w,w^2)$,
\item \label{hqclass3}
$(z,w) \mapsto \left( \frac{1}{z}, \frac{w^2}{z^2}, \frac{w}{z^2} \right)$,
\item \label{hqclass5}
$(z,w) \mapsto \left(
\frac{z^2+\sqrt{3}\,zw+w^2-z}{w^2+z+\sqrt{3}\,w - 1} ,
\frac{w^2+z-\sqrt{3}\,w - 1}{w^2+z+\sqrt{3}\,w - 1} ,
\frac{z^2-\sqrt{3}\,zw+w^2-z}{w^2+z+\sqrt{3}\,w - 1}
\right)$,
\item \label{hqclass4}
$(z,w) \mapsto \left(
\frac{\sqrt[4]{2}(zw-iz)}{w^2+\sqrt{2}\,iw+1} ,
\frac{w^2-\sqrt{2}\,iw+1}{w^2+\sqrt{2}\,iw+1} ,
\frac{\sqrt[4]{2}(zw+iz)}{w^2+\sqrt{2}\,iw+1}
\right)$,
\item \label{hqclass6}
$(z,w) \mapsto
\left( \frac{2w^3}{3z^2+1} ,
\frac{z^3+3z}{3z^2+1},
\sqrt{3}\frac{wz^2-w}{3z^2+1}
\right)$,
\item \label{hqclass7}
$(z,w) \mapsto (1,g(z,w),g(z,w))$ for an arbitrary CR function $g$.
\end{enumerate}
\end{thm}

Maps \eqref{hqclass1}, \eqref{hqclass2},
and \eqref{hqclass3}
come about from monomial maps (each component is a monomial)
from $Q(2,0)$ or $Q(1,1)$.
Classifying monomial maps leads to a problem in real-algebraic geometry,
which we discuss in \S~\ref{section:monomial}.

\begin{remark} \label{groupremark}
Interestingly, maps \eqref{faran3map}, \eqref{faran4map} of
Theorem~\ref{faranb2b3thm} and
\eqref{hqclass2} of Theorem~\ref{thm:hqclass} are group
invariant, although this fact is not used in this paper.
The map \eqref{hqclass2} of Theorem~\ref{thm:hqclass}
is one of the family of
CR maps of hyperquadrics invariant under a cyclic group obtained by D'Angelo~\cite{D:ginv}.
\end{remark}

\begin{remark} \label{regremark}
Faran proved his theorem with $C^3$ regularity.
To apply Faran's result on classification of planar maps we need
real-analytic CR maps.  For maps from the sphere to the $Q(2,1)$ hyperquadric
we get the map $(1,g,g)$ for an arbitrary CR function $g$.
Obviously, this map can have arbitrarily bad regularity.
\end{remark}

We will also
prove the following generalization of a theorem by Ji and
Zhang~\cite{JiZhang}.
The form of their maps was found by D'Angelo~\cite{D:prop1}
and Huang, Ji, Xu~\cite{HJX}.
The proof and the statement of the theorem by Ji and Zhang was more involved,
and covered only the case of source dimension 2.
A monomial map is a map whose every component is a monomial.  We do
not allow any monomial to have negative exponents.  The degree of a
rational map is the maximum degree of the numerator and the denominator, when
the map is
written in lowest terms.

The notion of spherical equivalence can be naturally extended to maps with
different target dimensions.  We will say two maps $f$ and $g$ with different
target dimensions (for example, target dimension of $f$ is smaller) are
\emph{spherically equivalent} if $f \oplus 0$ is spherically equivalent to
$g$ in the usual sense.

\begin{thm} \label{thm:deg2class}
Let $f \colon S^{2n-1} \to S^{2N-1}$, $n \geq 2$, be a rational CR map
of degree 2.  Then $f$ 
is spherically equivalent to a monomial map.

In particular, $f$ is equivalent to a map taking $(z_1,\ldots,z_n)$ to
\begin{multline} \label{alldeg2maps}
\bigl(
\sqrt{t_1} \, z_1 , 
\sqrt{t_2} \, z_2 , 
\ldots ,
\sqrt{t_n} \, z_n , 
\sqrt{1-t_1} \, z_1^2 , 
\sqrt{1-t_2} \, z_2^2 , 
\ldots ,
\sqrt{1-t_n} \, z_n^2 , 
\\
\sqrt{2-t_1-t_2} \, z_1z_2 , 
\sqrt{2-t_1-t_3} \, z_1z_3 , 
\ldots ,
\sqrt{2-t_{n-1}-t_n} \, z_{n-1}z_n \bigr) , 
\\
0 \leq t_1 \leq t_2 \leq \ldots \leq t_n \leq 1,
\quad (t_1,t_2,\ldots,t_n) \not= (1,1,\ldots,1) .
\end{multline}
Furthermore, all maps of the form \eqref{alldeg2maps} are mutually
spherically inequivalent for different parameters $(t_1,\ldots,t_n)$.
\end{thm}

It is worthwhile and generally more convenient
to see what the maps \eqref{alldeg2maps} look like in
more abstract language.  D'Angelo~\cite{DAngelo:CR} has shown that any
degree-two polynomial map taking the origin to the origin
can be abstractly written as
\begin{equation} \label{abstmaps}
Lz \oplus ( \sqrt{I-L^*L} z ) \otimes z,
\end{equation}
where $L$ is any linear map such that
$I-L^*L$ is positive semi-definite, and $z$ is the identity map in
nonhomogeneous coordinates.
A monomial map then corresponds to taking a diagonal $L$.  We can make $L$
have nonnegative entries, and by permuting the variables we can sort the
diagonal entries of $L$.  Hence our result could be stated as follows.
\emph{Every nonconstant degree-two rational CR map of spheres is spherically
equivalent to exactly one map of the form \eqref{abstmaps} where
$L$ is diagonal with nonnegative diagonal entries sorted by size, such
that $I-L^*L$ also has nonnegative entries.}
With the result
of D'Angelo it is obvious that \eqref{abstmaps} and
hence \eqref{alldeg2maps} is
exhaustive, that is, gives all monomial maps.  We will give another proof.

The main point is that all degree-two CR maps of spheres are spherically
equivalent to monomial maps.
Faran, Huang, Ji, and Zhang~\cite{FHJZ} have shown that for $n=2$, all
degree-two rational CR maps of spheres are equivalent to polynomial maps,
so we
improve on their result by allowing the source dimension to be arbitrary
and showing the maps are in fact monomial, not just polynomial.
An explicit example is also given in \cite{FHJZ} of a rational
degree-three CR map of $S^3$ to $S^7$ that is not spherically equivalent to
any polynomial map.  Therefore, the result is in some sense optimal.  See
also Proposition~\ref{exists3deg} for a slightly weaker statement.

The family of maps \eqref{alldeg2maps} in
the simplest case $n=2$ are the following
maps taking $S^3$ to $S^9$.
\begin{multline}
(z,w) \mapsto
\bigl(
\sqrt{s} \, z,
\sqrt{t} \, w,
\sqrt{1-s} \, z^2,
\sqrt{2-s-t} \, zw,
\sqrt{1-t} \, w^2 \bigr) , \\
0 \leq s \leq t \leq 1, \quad (s,t) \not= (1,1) .
\end{multline}
This family appeared already in the work of Wono~\cite{Wono},
who classified all monomial maps from $S^3$ to $S^9$.
Note that $s=0$ corresponds to maps of type (I) in 
\cite{JiZhang} and $t=1$ corresponds to type (IIA),
and the other cases correspond to
type (IIC)\@.
As there is only one two-dimensional family, it is not hard to see that
this two-dimensional
family corresponds to the maps (IIC) with
the maps of type (I) and (IIA) being on the ``boundary'' of the family.

The author would like to thank James Faran for explaining and
discussing the proof of his theorem, and
John D'Angelo for many discussions
on the topic and for suggestions that greatly improved this manuscript.
The author also thanks the referee for many suggestions and improvements.
The author would like to acknowledge
AIM for holding a workshop on the topic of complexity in CR
geometry in the
summer of 2006,
and MSRI for holding a workshop in 2005,
where the author was first exposed to the topic of proper maps of balls.
Finally, the author would like to acknowledge the Espresso Royale Cafe,
without whose coffee this work would not have been possible.


\section{Hermitian forms} \label{section:hermforms}

Let $\langle \cdot , \cdot \rangle$ denote the
standard pairing
$\langle z , w \rangle =
z_0 \bar{w}_0 + z_1 \bar{w}_1 + \cdots + z_n \bar{w}_n$.
Let $A$ be a Hermitian matrix, then a Hermitian form is simply
\begin{equation}
\langle A z , z \rangle ,
\end{equation}
where $z \in \C^{n+1}$.
We will often talk of the zero sets of Hermitian forms and so let us
define
\begin{equation}
\bV_A :=
\left\{ z \in \C^{n+1} \mid
\langle A z , z \rangle = 0 \right\} .
\end{equation}

In CR geometry we often think of a real polynomial
as a Veronese map composed with a Hermitian form.  Let $p$ be a
polynomial in $(z,w) \in \C^{n+1} \times \C^{n+1}$:
\begin{equation} \label{eqpdef}
p(z,w) = \sum_{\alpha\beta} a_{\alpha\beta} z^\alpha w^\beta .
\end{equation}
Suppose that
$p$ is bihomogeneous of bidegree $(d,d)$,
that is, $p(tz,w) = p(z,tw) = t^d p(z,w)$.
The polynomial is said to be \emph{Hermitian symmetric} if 
$p(z,\overline{w}) = \overline{p(w,\bar{z})}$.  In other words,
$a_{\alpha\beta} = \overline{a_{\beta\alpha}}$, that is,
the matrix with the entries $a_{\alpha\beta}$ is Hermitian.
It is not hard to see that
$p$ is Hermitian symmetric
if and only if $p(z,\bar{z})$ is real valued.

Let $\sZ = \sZ_d$ be the degree $d$ Veronese map
\begin{equation}
(z_0,\ldots,z_n) \overset{\sZ}{\mapsto}
(z_0^d,z_0^{d-1}z_1,\ldots,z_n^d),
\end{equation}
that is, the map
whose components are all the degree $d$ monomials.
We can think of $p$ as
\begin{equation}
p(z,\bar{z}) = \langle A \sZ, \sZ \rangle ,
\end{equation}
where $A = [ a_{\alpha\beta} ]_{\alpha\beta}$ is the
matrix of coefficients from \eqref{eqpdef}.
By \emph{signature}
of $p$ we will mean the signature of $A$.
Writing $A$ as a sum of rank one matrices we will see below
that $p$ is the composition
of a diagonal Hermitian form with the same signature as $A$ and a homogeneous
holomorphic map of $\C^{n+1}$ to some $\C^{N+1}$.  That is, we obtain a map
taking the zero set of $p$ to a hyperquadric.

Let $f$ be
a rational map of $\bP^{n}$ to $\bP^{N}$.  In
homogeneous coordinates, $f$ is given by
$N+1$ homogeneous polynomials $f = f_0,f_1,\ldots,f_N$.
We wish to formulate what it means for $f$ to take
$\bV_J \subset \bP^{n}$ to $\bV_V \subset \bP^{N}$ for 
two Hermitian matrices $J$ and $V$.
We can simply plug $f$ into the equation for $\bV_V$ to obtain
\begin{equation}
\langle V f(z) , f(z) \rangle = 0 .
\end{equation}
Then $f$ takes
$\bV_J$ to $\bV_V$ if and only if there exists a bihomogeneous real
polynomial $q$ such that
\begin{equation}
\langle V f(z) , f(z) \rangle = q(z,\bar{z}) \langle J z , z \rangle .
\end{equation}
To classify $f$ that take $\bV_J$ to $\bV_V$
we first need to identify those $f$ that differ by an automorphism
preserving $\bV_V$.

It will be useful to represent $\langle V f(z) , f(z) \rangle$
slightly differently.  Suppose that $f$ is of degree $d$.
Let $\sZ$ be the degree $d$ Veronese map.
Then
we can write $f(z) = F \sZ$ for some complex matrix $F$.  We write
\begin{equation}
\langle V f(z) , f(z) \rangle
=
\langle V F \sZ , F \sZ \rangle
=
\langle F^* V F \sZ , \sZ \rangle .
\end{equation}

\begin{lemma} \label{formprop}
Let $V$ be a nonsingular Hermitian matrix.  Let $f$ and $g$ be homogeneous
polynomial maps of $\C^{n+1}$ to $\C^{N+1}$ such that
\begin{equation}
\langle V f(z) , f(z) \rangle = \langle V g(z) , g(z) \rangle .
\end{equation}
Suppose that $f$ has linearly independent components.
Then
\begin{equation}
g(z) = C f(z)
\end{equation}
for some invertible matrix $C$ such that $C^* V C = V$.
\end{lemma}

\begin{proof}
As explained above, write
$f(z) = F \sZ$ for some complex matrix $F$ with linearly independent rows.
Similarly, $g(z) = G \sZ$.
Thus
$\langle V f(z) , f(z) \rangle = \langle V g(z) , g(z) \rangle$ implies
that 
\begin{equation}
\langle ( F^* V F - G^* V G) \sZ, \sZ \rangle = 0.
\end{equation}
We have a real polynomial that is identically zero and so its coefficients
are zero.  Thus $F^* V F = G^* V G$.  By
permuting the monomials, we could suppose that
$F = [ \, F_1 ~ F_2 \, ]$, where $F_1$ is an invertible matrix.  We write
$G = [ \, G_1 ~ G_2 \, ]$, where $G_1$ is square.
Now $F^* V F = G^* V G$ implies
$F_1^* V F_1 = G_1^* V G_1$.  Let $C = G_1
F_1^{-1}$.  Clearly $C^* V C = V$ and $G_1 = C F_1$.  So if $G_2 = C F_2$,
then we are finished.
$F^* V F = G^* V G$ also implies
$F_1^* V F_2 = G_1^* V G_2$.  Replace $G_1$ with $C F_1$ and obtain
$F_1^* V F_2 = F_1^* V C^{-1} G_2$.  As $F_1$ and $V$ are invertible
we get $G_2 = C F_2$ and thus $g(z) = C f(z)$.
\end{proof}

Lemma~\ref{formprop} says that if 
$\langle V f(z) , f(z) \rangle = \langle V g(z) , g(z) \rangle$, then
$f$ and $g$ are equal up to a linear map of the target space preserving
the form defined by $V$.  Note that these linear maps (up to a scalar multiple)
correspond exactly to
linear fractional
transformations that take $\bV_V$ to
itself and preserve the sides of $\bV_V$.
Since we are working with homogeneous coordinates, we must also always
consider the possibility
$\langle V f(z) , f(z) \rangle = \lambda \langle V g(z) , g(z) \rangle$
for $\lambda > 0$.  We can then rescale $f$ or $g$ and use the proposition.
Furthermore,
if the number of positive and negative eigenvalues of $V$ are equal, then
there exists a linear map that takes the form corresponding
to $V$ to the form corresponding to $-V$.  In this case, we can have
$\langle V f(z) , f(z) \rangle = - \langle V g(z) , g(z) \rangle$
and still $f$ and $g$ differ by an automorphism of the target
and take the same set to $\bV_V$, but one swaps sides
and one does not.

For the source and target we
are mostly interested in hyperquadrics and spheres.  When $V$ is Hermitian,
the set $\bV_V$
is equivalent to a hyperquadric by an
automorphism of $\bP^N$.
For the hyperquadric
$Q(a,b)$
we let $V$ be the matrix with $a$ ones and $b+1$ negative ones
on the diagonal.
We have the
following corollary.

\begin{cor}
Let $f$ and $g$ be rational CR maps of $Q(a,b)$ to $Q(c,d)$, and
let $V$ be a Hermitian form defining $Q(c,d)$.
Let $\hat{f}$ and $\hat{g}$ be the corresponding homogeneous polynomial
maps.
Suppose that the components of $\hat{f}$ are linearly independent.
Then $f$ and
$g$ are Q-equivalent if and only if there exists a $\chi \in U(a,b)$
and
$\lambda > 0$ if $c \not= d+1$
or $\lambda \in \R \setminus \{0\}$ if $c=d+1$, such that
\begin{equation} \label{cor2eq}
\langle V \hat{f}(\chi z) , \hat{f}(\chi z) \rangle = \lambda \langle V
\hat{g}(z) , \hat{g}(z) \rangle \quad \text{ for all $z$}.
\end{equation}
\end{cor}

Consequently, by working in projective space and with 
Hermitian forms rather than the maps themselves, we reduce the Q-equivalence
problem to an equivalence problem using only the group of automorphisms of
the source.

Now 
suppose we start with any real polynomial 
$\langle B \sZ_d , \sZ_d \rangle$ 
for some Hermitian matrix $B$.
We write $B$ as a sum of rank one matrices.  As
$B$ is Hermitian, this can be done in the following way.  We take the positive
eigenvalues $\lambda_1,\ldots,\lambda_k$ and the negative eigenvalues
$-\lambda_{k+1},\ldots,-\lambda_r$.  We ignore zero eigenvalues if
present.  We let $v_1,\ldots,v_r$ be an orthonormal set of corresponding
eigenvectors.
\begin{equation}
B = \sum_{j=1}^k \lambda_j v_j v_j^* - \sum_{j=k+1}^r \lambda_j v_j v_j^* .
\end{equation}
Now we define $f_j(z) = \sqrt{\lambda_j} \, v_j^* \sZ$.  We see that
\begin{equation}
\langle V f(z), f(z) \rangle = 
\langle B \sZ, \sZ \rangle ,
\end{equation}
where $V$ is the form with $k$ ones and $r-k$ negative ones on the diagonal.

Hence any real polynomial corresponds to a holomorphic map $f \colon \bP^n \to \bP^N$
taking the zero set of the polynomial to a hyperquadric with the same
signature as the coefficient matrix $B$.
Therefore to classify all maps taking one hyperquadric to another
hyperquadric up to Q-equivalence, we need only classify real polynomials that
vanish on the source hyperquadric and have the correct signature.

Note that we will generally
scale all the $v_j$ by the same number to make the expressions easier to work
with.  Other methods can be employed to construct
Q-equivalent $f$ by writing $B$ differently as a sum
of rank one matrices, but the procedure outlined above is what we will use.
Any other
method will produce the same number of positive and the same number of
negative terms, assuming the components are linearly independent.

If we have a Hermitian matrix $B$ such that $\langle B \sZ_d , \sZ_d \rangle$ 
is zero on $\bV_J$, then
\begin{equation}
\langle B \sZ_d , \sZ_d \rangle =
\langle A \sZ_{d-1} , \sZ_{d-1} \rangle
\langle J z , z \rangle ,
\end{equation}
for some Hermitian matrix $A$.  We take an automorphism preserving the set
defined by $J$, that is, a matrix $X$ such that $X^* J X = J$.  To find all maps
equivalent by an automorphism of the source we compute the canonical form of
\begin{equation}
\langle A \sZ_{d-1} (Xz) , \sZ_{d-1} (Xz) \rangle
\langle J z , z \rangle 
=
\langle \hat{X}^* A \hat{X} \sZ_{d-1} , \sZ_{d-1} \rangle
\langle J z , z \rangle ,
\end{equation}
where $\hat{X}$ is the matrix defined by $\sZ_{d-1}(Xz) = \hat{X} \sZ_{d-1}(z)$.
Hence, we will look for a canonical form of the pair $(J,A)$ under $*$-conjugation
by $\hat{X}$ and $X$.  If $J$ has the same number of positive and negative
eigenvalues, we would have to also consider linear maps such that
$X^* J X = -J$.
When $d=2$, then $\hat{X} = X$ and matters become simpler
and we describe the method in more detail.

In general, we could start with any algebraic source manifold,
for example defined
by the zero set of the real polynomial
$\langle C \sZ_k, \sZ_k \rangle$.  Then the existence of a map $f$ taking
the manifold to a hyperquadric means that there exists a Hermitian matrix $A$
such that
\begin{equation}
\langle B \sZ_d, \sZ_d \rangle
=
\langle A \sZ_{d-k}, \sZ_{d-k} \rangle
\langle C \sZ_k, \sZ_k \rangle ,
\end{equation}
where $B$ is the form resulting from composing the map $f$ with the defining
equation of the target hyperquadric.
Classifying such maps is then equivalent to finding
normal forms for the pair $(C,A)$ under the automorphism group
fixing $C$.
In this paper we study the case when the
source manifold is a hyperquadric.


\section{Degree 2 maps} \label{section:deg2maps}

Let us consider degree-two maps.  Let $f$ be a degree-two map taking
a hyperquadric to a hyperquadric.  Then for the proper $V$ and $J$
we have
$\langle V f(z) , f(z) \rangle = q(z,\bar{z}) \langle J z , z \rangle$.
The polynomial $q$ is real valued of bidegree $(1,1)$ and can, therefore,
be also written
as a Hermitian form.  We write $f(z) = F \sZ$ as before.  We obtain
\begin{equation}
\langle F^* V F \sZ , \sZ \rangle =
\langle A z , z \rangle \langle J z , z \rangle
\end{equation}
for some Hermitian matrix $A$.

If on the other hand we start with
$\langle A z , z \rangle \langle J z , z \rangle$ and we multiply, we will
get some Hermitian matrix $B$ such that
$\langle B \sZ , \sZ \rangle = \langle A z , z \rangle \langle J z , z
\rangle$.  If $B$ and $V$ have the same signature, we can find a matrix $F$
such that $B = F^* V F$ by linear algebra as explained in \S~\ref{section:hermforms}.

Therefore, by going through all possible Hermitian
matrices $A$ we get all degree-two maps of $\bV_J$ to all hyperquadrics.
We want to restrict only 
to those matrices $A$ where the resulting matrix $B$ has the right signature.

We assume that the components of $f$ are linearly
independent.  This limitation turns out not to be a problem for CR maps of
spheres, but it is a problem for CR maps of hyperquadrics where we miss maps
such as $(1,g,g)$, but that is essentially all we miss.

If $Q(a,b)$ is the source hyperquadric (corresponding to $J$)
and we take an automorphism $\chi$
of $Q(a,b)$, then $\chi$ is represented by a matrix $X$ such that
$X^* J X = J$.  Therefore, 
\begin{equation}
\langle A Xz , Xz \rangle \langle J Xz , Xz \rangle
=
\langle X^* A X z , z \rangle \langle J z , z \rangle .
\end{equation}
Therefore, we first find all the canonical forms for a pair Hermitian
matrices under the equivalence $(A,B) \sim (C,D)$ whenever there
exist a nonsingular matrix $X$ such that $X^*AX = C$ and $X^*BX = D$
(simutaneous $*$-congruence).
Next we collect those canonical pairs where one of the matrices is
$*$-congruent to $J$.  Such pairs therefore give canonical forms of
degree-2 CR maps of hyperquadrics from the $J$ hyperquadric
up to Q-equivalence.
In the special case when the number of positive and negative
eigenvalues of $J$ is the same (i.e.\ $a=b+1$), then we also need to
consider $X$ such that $X^* J X = -J$.  However, we will usually assume
that $J$ defines the sphere and $a+b\geq 2$.

The problem of the classification of a pair of matrices
has a long history going back to Kronecker.  The first
results for Hermitian matrices and $*$-congruence
were proved in the 1930s by several authors
independently.  See the survey by
Lancaster and Rodman~\cite{LancasterRodman}.
We could use these older results.
For convenience, we will
use a recent paper by
Horn and Sergeichuk~\cite{HornSergeichuk},
whose canonical form is easy to work with.
They also give an algorithm for computing the canonical form, so we also
have an algorithm for producing a normal form for maps of hyperquadrics.
That is, to decide if two degree-2
CR maps of hyperquadrics are Q-equivalent, we first find
the corresponding Hermitian forms.  Then we
follow the procedure by Horn and Sergeichuk to produce a normal form of
the matrices and then simply check and see if the normal forms are the same
(up to a multiple of course).
This method works only if the components of the maps are
linearly independent.  However, it is not hard to reduce to this case in
general.

Before we can give the result of Horn and Sergeichuk, we first need to define
the building block matrices used for their canonical form.
Let us define $M_n$ as the $n \times n$ matrix with ones on the subdiagonal
and the superdiagonal, and similarly
$N_n$ as the $n \times n$ matrix with negative ones on the subdiagonal
and ones on the super diagonal.
\begin{equation}
M_n :=
\begin{bmatrix}
0 & 1 & & 0 \\
1 & 0 & \ddots &  \\
& \ddots & \ddots & 1 \\
0 & & 1 & 0 \\
\end{bmatrix}
, \quad
N_n :=
\begin{bmatrix}
0 & 1 & & 0 \\
-1 & 0 & \ddots &  \\
& \ddots & \ddots & 1 \\
0 & & -1 & 0 \\
\end{bmatrix} .
\end{equation}
Let $J_n(\lambda)$ be the $n \times n$ Jordan block with eigenvalue $\lambda$
($\lambda$ on the diagonal, ones on the superdiagonal).  Finally define
\begin{equation}
\Delta_n(\alpha,\beta) := 
\begin{bmatrix}
0 & & & & \alpha \\
& & & \alpha & \beta \\
& & \iddots & \iddots & \\
& \alpha & \beta & & \\
\alpha & \beta & & & 0 \\
\end{bmatrix} .
\end{equation}
We can now give the classification theorem, or at least the part of the
theorem useful for us.  $I_n$ is the $n \times n$ identity matrix.

\begin{thm}[Horn-Sergeichuk~\cite{HornSergeichuk}]
Let $A$, $B$ be a pair of Hermitian matrices.  They are simultaneously
$*$-congruent to a direct sum of blocks of the following four types,
determined uniquely up to permutation.
\begin{enumerate}[(i)]
\item $\bigl( M_n, iN_n \bigr)$,
\item $\pm \bigl( \Delta_n(1,0),\Delta_n(\alpha,1) \bigr)$, $\alpha \in \R$,
\item $\pm \bigl( \Delta_n(0,1),\Delta_n(1,0) \bigr)$,
\item $\left(
\begin{bmatrix}
0 & I_n \\
I_n & 0
\end{bmatrix} ,
\begin{bmatrix}
0 & J_n(\alpha+i\beta)^* \\
J_n(\alpha+i\beta) & 0
\end{bmatrix}
\right)$, $\alpha,\beta \in \R$, $\alpha+\beta i \not=i$, $\beta > 0$.
\end{enumerate}
\end{thm}

When we write $\pm \bigl( \Delta_n(1,0),\Delta_n(\alpha,1) \bigr)$ we mean
the pair $\bigl( \Delta_n(1,0),\Delta_n(\alpha,1) \bigr)$ or the pair
$\bigl(\Delta_n(-1,0), \Delta_n(-\alpha,-1) \bigr)$.  Similarly for $\pm
\bigl( \Delta_n(0,1),\Delta_n(1,0) \bigr)$.

Let us use this classification theorem.  Let us study the case
when the source hyperquadric is a sphere.
Take a degree-two rational CR map $f \colon S^{2n-1} \to Q(c,d)$,
and obtain the pair of Hermitian matrices
$(J,A)$.
We will put the pair $(J,A)$ into canonical
form.
If we can diagonalize $J$ and $A$ simultaneously by $*$-congruence, then
the map is equivalent to a monomial map.  That is, the map in homogeneous
coordinates is monomial.  If the map is monomial (not allowing for negative
exponents) in nonhomogeneous coordinates remains to be seen.
Suppose that for $k=0,\ldots,N$ the $k$th component of the homogenized
map $\hat{f}$
is $c_k z^{\alpha_k}$ for some degree-two multi-index $\alpha_k$.
Now write
$\langle V \hat{f}(z), \hat{f}(z) \rangle = \langle A z, z \rangle \langle J z, z \rangle$.
As $A$ is not identically zero, there exists a $j = 0,\ldots,n$
such that when we set $z_j = 1$ and
$z_m = 0$ for all $m \not= j$ we obtain
$\langle A z, z \rangle \not= 0$.  As $J$ is invertible
we get
$\langle J z, z \rangle \not= 0$ and thus
$\langle V \hat{f}(z), \hat{f}(z) \rangle \not= 0$.  As
\begin{equation}
\langle V \hat{f}(z), \hat{f}(z) \rangle =
\sum_{k=0}^N \pm \abs{c_k}^2 \abs{z^{\alpha_k}}^2 ,
\end{equation}
we see that for at least one $k$ we have
$z^{\alpha_k} = z_j^2$, that is, $\hat{f}$ contains a pure monomial
term.  Therefore in the nonhomogeneous coordinates obtained by
$z_j = 1$, we get a monomial (no negative exponents) CR map from $Q(a,b)$
to $Q(c,d)$, where $Q(a,b)$ is equivalent to $S^{2n-1}$.
We will deal with the classification of monomial maps for each problem
separately in later sections.

Therefore,
let us see what
happens in the case when we cannot simultaneously diagonalize $J$ and $A$.
From now on, assume that at least
one block in the canonical form is larger than $1 \times 1$.

As $J$ defines the sphere,
we note that all blocks
except one in the canonical form for $J$ must be $1 \times 1$ blocks.
Indeed, all the possible blocks of larger size have at least 1 negative
eigenvalue, and hence there can only be one of these.  Furthermore, $J$
has no zero eigenvalues.
Such blocks that have at most 1 negative eigenvalue and no zero eigenvalues
are
$M_2$, $\pm \Delta_2(1,0)$, $\Delta_3(1,0)$,
$\left[ \begin{smallmatrix} 0 & I_1 \\ I_1 & 0 \end{smallmatrix} \right]$.
All these blocks are $2 \times 2$ or $3 \times 3$.
It will be sufficient to consider only source dimension $n=2$ in all that
follows.  So let us suppose that $J$ and $A$ are both
$3 \times 3$, and let us compute all the normal forms for the
Hermitian forms that arise from degree-two maps.

Note that 
$\left[ \begin{smallmatrix} 0 & I_1 \\ I_1 & 0 \end{smallmatrix} \right] =
M_2$, and $iN_2$ can be obtained by swapping variables and letting
$\alpha+i\beta = i$
in
the block 
$\left[ \begin{smallmatrix} 0 & \alpha-i\beta \\ \alpha+i\beta & 0 \end{smallmatrix} \right]$.
So we need not consider the block $M_2$, but just need to allow
$\alpha+i\beta = i$.

Hence let us compute.  First let us take the situation corresponding to
the block $\Delta_2$ in $J$.  That is we have the following canonical form 
for $(J,A)$
\begin{equation}
\left(
\begin{bmatrix}
       1 &        0 &        0\\
       0 &        0 &        1\\
       0 &        1 &        0
\end{bmatrix}
,
\begin{bmatrix}
  \alpha &        0 &        0\\
       0 &        0 &    \beta\\
       0 &    \beta &        1
\end{bmatrix}
\right) .
\end{equation}
For $\alpha, \beta \in \R$.
Take
$\sZ = (z_0^2,z_0z_1,z_0z_2,z_1^2,z_1z_2,z_2^2)$.
The form corresponding to the product
$\langle A z, z \rangle \langle J z, z \rangle  = \langle B \sZ , \sZ
\rangle$ is
\begin{equation} \label{form1B}
B=
\begin{bmatrix}
  \alpha &        0 &        0 &        0 &        0 &        0\\
       0 &        0 & \alpha+\beta &        0 &        0 &        0\\
       0 & \alpha+\beta &        1 &        0 &        0 &        0\\
       0 &        0 &        0 &        0 &        0 &    \beta\\
       0 &        0 &        0 &        0 &   2\beta &        1\\
       0 &        0 &        0 &    \beta &        1 &        0
\end{bmatrix} .
\end{equation}
The forms for the related canonical pair
$\left(
\left[
\begin{smallmatrix}
       1 &        0 &        0\\
       0 &        0 &        -1\\
       0 &        -1 &        0
\end{smallmatrix}
\right]
,
\left[
\begin{smallmatrix}
  \alpha &        0 &        0\\
       0 &        0 &    \beta\\
       0 &    \beta &        -1
\end{smallmatrix}
\right]
\right)$ are similar.

Next pair to consider is
\begin{equation}
\left(
\begin{bmatrix}
       0 &        0 &        1\\
       0 &        1 &        0\\
       1 &        0 &        0
\end{bmatrix}
,
\begin{bmatrix}
       0 &        0 &   \alpha\\
       0 &   \alpha &        1\\
  \alpha &        1 &        0
\end{bmatrix}
\right) ,
\end{equation}
for $\alpha \in \R$.
After multiplying the forms we get
\begin{equation} \label{form2B}
B = \begin{bmatrix}
       0 &        0 &        0 &        0 &        0 &   \alpha\\
       0 &        0 &        0 &        0 &  2\alpha &        1\\
       0 &        0 &  2\alpha &        0 &        1 &        0\\
       0 &        0 &        0 &   \alpha &        1 &        0\\
       0 &  2\alpha &        1 &        1 &        0 &        0\\
  \alpha &        1 &        0 &        0 &        0 &        0
\end{bmatrix} .
\end{equation}

Finally we get to the canonical pair
\begin{equation}
\left(
\begin{bmatrix}
       1 &        0 &        0\\
       0 &        0 &        1\\
       0 &        1 &        0
\end{bmatrix}
,
\begin{bmatrix}
  \alpha &        0 &        0\\
       0 &        0 & \beta-i\gamma\\
       0 & \beta+i\gamma &        0
\end{bmatrix}
\right) ,
\end{equation}
for $\alpha, \beta \in \R$ and $\gamma > 0$.  We note that we can just rescale to have $\gamma = 1$
if needed.
After multiplying the forms we get
\begin{equation} \label{form3B}
B = \begin{bmatrix}
  \alpha &        0 &        0 &        0 &        0 &        0\\
       0 &        0 & \alpha+\beta-i\gamma &        0 &        0 &        0\\
       0 & \alpha+\beta+i\gamma &        0 &        0 &        0 &        0\\
       0 &        0 &        0 &        0 &        0 & \beta-i\gamma\\
       0 &        0 &        0 &        0 & 2\beta &        0\\
       0 &        0 &        0 & \beta+i\gamma &        0 &        0
\end{bmatrix} .
\end{equation}

Once we have the matrix $B$, we write $B$ as a sum of rank one matrices to obtain a representative
of the class of maps given by $B$.
We must also still change variables to put $J$
into the form
$
\left[
\begin{smallmatrix}
       1 &        0 &        0\\
       0 &        1 &        0\\
       0 &        0 &        -1
\end{smallmatrix}
\right]$, so that we always have the same ball as our source.
We could also change
variables first and then write $B$ as a sum of rank one matrices.


\section{Degree 2 maps of \texorpdfstring{$S^{2n-1}$ to $S^{2N-1}$}{S(2n-1) to S(2N-1)}} \label{section:deg2mapsspheres}

In this section we will prove that all degree-two maps of spheres are
Q-equivalent (that is, spherically equivalent) to a monomial map, and then we
will give the classification of those monomial maps.
First, let us prove that all degree-two maps are equivalent to monomial maps.

\begin{lemma} \label{lemma:mapsmonomial}
Let $f \colon S^{2n-1} \to S^{2N-1}$, $n \geq 2$, be a rational CR map
of degree 2.  Then $f$ 
is spherically equivalent to a monomial map.
\end{lemma}

\begin{proof}
Suppose we have a degree-two map of spheres.  That is, we
have a Hermitian matrix $A$, such that
\begin{equation}
\langle A z , z \rangle \langle J z , z \rangle
\end{equation}
has only one negative eigenvalue.
$J = \operatorname{diag}(1,1,\ldots,1,-1)$ is the matrix
giving
us a sphere, that is, it has ones along the diagonal except
the last diagonal element, which is $-1$.
Unlike above, we start in arbitrary source dimension,
but we will see that considering
$n=2$ is sufficient.

If we can diagonalize $A$ and $J$ simultaneously by
$*$-congruence, we are essentially done.  The map will be
equivalent to a monomial map if we, for a moment, allow some
exponents to be negative in nonhomogeneous coordinates.  We have already
remarked that the map has no negative exponents in some set of nonhomogeneous
coordinates,
but we wish it to be so in the nonhomogeneous coordinates where $\bV_J$ is
$S^{2n-1}$.

Negative exponents in nonhomogeneous coordinates mean
that the monomial corresponding to the negative
eigenvalue in the form for $f$ is divisible by one of the variables
corresponding to a positive eigenvalue for $J$.  That is, take the
homogeneous coordinates $z = (z_0,z',z_n) \in \C \times \C^{n-1} \times \C =
\C^{n+1}$ and write
\begin{equation}
\norm{f'(z)}^2 - \abs{z^\alpha}^2 = q(z,\bar{z}) \bigl(\abs{z_0}^2 + \norm{z'}^2
- \abs{z_n}^2 \bigr) .
\end{equation}
for some bihomogeneous real polynomial $q$.  We denote by $f'$ the part of
the monomial map corresponding to positive eigenvalues.  We can assume that
the components of $f'$ together with the monomial $z^\alpha$ have no common
factor.  If $z^\alpha = z_n^d$ we are done.

Therefore, assume that after perhaps
renaming of variables, $z^\alpha$ is divisible by $z_0$.
We set $z_0 = 0$ to obtain
$\norm{f'(0,z',z_n)}^2 = q(0,z',z_n,0,\bar{z}',z_n) \bigl(\norm{z'}^2
- \abs{z_n}^2 \bigr)$.  $f'(0,z',z_n)$ vanishes on a real hypersurface
and hence is
identically zero.  This means that $f'$ is divisible by $z_0$, that is, the
map is not given in lowest terms, a contradiction.  So if we can diagonalize,
the map is monomial, with the usual understanding that in nonhomogeneous
coordinates
where the source is the sphere, no exponent is negative.

We focus on the case where we cannot diagonalize.  Note that we
are seeking a contradiction.
As above,
we find the canonical form for the pair $(J,A)$ and we
assume that there is some block
larger than $1 \times 1$ in the canonical form for $J$.
Again we note that all such blocks for $J$ that have at most one negative
eigenvalue are either $2 \times 2$ or $3 \times 3$.  Therefore, if we prove
the result for source dimension 2, then we are done; if there were a map
not equivalent to a monomial map for higher source dimension, then we can
set all but 3 variables that correspond to
$1 \times 1$ blocks to zero.  We obtain a CR map of spheres
with source dimension 2
that is not equivalent to a monomial map, as the surviving
$2 \times 2$ or $3 \times 3$ block is a canonical block and therefore
the corresponding block in $A$ is not simultaneously diagonalizable.
If we can show that no such map can exist, then all blocks in the
canonical form of the pair $(J,A)$ must have been $1 \times 1$
and therefore the matrices must have been simultaneously diagonalizable.
Hence assume that
$J$ and $A$ are $3 \times 3$ just as they are above.

We have already computed the matrices $B$, resulting from
$\langle A z, z \rangle \langle J z, z \rangle  = \langle B \sZ , \sZ
\rangle$.  Let us consider these
matrices one by one.

First, let us take the situation corresponding to
the block $\Delta_2$ in $J$.  That is, for $(J,A)$
consider canonical form 
$
\left(
\left[
\begin{smallmatrix}
       1 &        0 &        0\\
       0 &        0 &        1\\
       0 &        1 &        0
\end{smallmatrix}
\right]
,
\left[
\begin{smallmatrix}
  \alpha &        0 &        0\\
       0 &        0 &    \beta\\
       0 &    \beta &        1
\end{smallmatrix}
\right]
\right)
$.
The corresponding $B$ is computed in \eqref{form1B}.
If $\alpha+\beta \not= 0$, then
this form clearly has at least two negative eigenvalues (notice the block
structure). 
If $\beta = -\alpha \not= 0$, then it is also not hard to see that there must
still be at least two negative eigenvalues.  When $\beta = \alpha = 0$,
$A$ has rank one, and we have a first degree map.
Similar calculations show the same result for the related canonical
form
$\left(
\left[
\begin{smallmatrix}
       1 &        0 &        0\\
       0 &        0 &        -1\\
       0 &        -1 &        0
\end{smallmatrix}
\right]
,
\left[
\begin{smallmatrix}
  \alpha &        0 &        0\\
       0 &        0 &    \beta\\
       0 &    \beta &        -1
\end{smallmatrix}
\right]
\right)$.

Next let us consider the pair
$
\left(
\left[
\begin{smallmatrix}
       0 &        0 &        1\\
       0 &        1 &        0\\
       1 &        0 &        0
\end{smallmatrix}
\right]
,
\left[
\begin{smallmatrix}
       0 &        0 &   \alpha\\
       0 &   \alpha &        1\\
  \alpha &        1 &        0
\end{smallmatrix}
\right]
\right)$.
The corresponding $B$ is computed in \eqref{form2B}.
This form also always has at least two negative eigenvalues.  To see this fact,
notice that the form has a zero eigenvalue only when $\alpha = 0$.  Hence,
it is enough to check the signature of the matrix for $\alpha  = 0$.

Finally we get to the canonical pair
$\left(
\left[
\begin{smallmatrix}
       1 &        0 &        0\\
       0 &        0 &        1\\
       0 &        1 &        0
\end{smallmatrix}
\right]
,
\left[
\begin{smallmatrix}
  \alpha &        0 &        0\\
       0 &        0 & \beta-i\gamma\\
       0 & \beta+i\gamma &        0
\end{smallmatrix}
\right]
\right)$,
for $\gamma > 0$.
The corresponding $B$ is computed in \eqref{form3B}.
Again, because of the block structure and the fact that $\gamma \not= 0$,
we easily see that this form must have at least two negative eigenvalues.

We are done.  We have dealt with all the canonical forms and have shown
that the canonical form for $(J,A)$ must contain only $1 \times 1$ blocks,
and hence it must be monomial.  So the lemma is proved.
\end{proof}

\begin{remark}
Do note that when $n = 1$, the theorem is not true.  It is not hard to create
counterexamples.  For example,
we can use the techniques above to construct a counterexample.
Take the canonical pair
$\left(
\Bigl[
\begin{smallmatrix}
       0 &        1\\
       1 &        0
\end{smallmatrix}
\Bigr]
,
\Bigl[
\begin{smallmatrix}
    0 &    i \\
   -i &    0
\end{smallmatrix}
\Bigr]
\right)$.  With $\sZ = (z_0^2,z_0z_1,z_1^2)$, the form $B$ becomes
\begin{equation}
\begin{bmatrix}
            0 &        0 &             i \\
            0 &        0 &             0 \\
           -i &        0 &             0
\end{bmatrix}
.
\end{equation}
We find the eigenvalues $1$ and $-1$ and the corresponding
orthonormal eigenvectors
$[ 1 ~ 0 ~ {-i} ]^T / \sqrt{2}$ and 
$[ 1 ~ 0 ~ i ]^T / \sqrt{2}$.  To construct the map we scale both
eigenvectors by the same number to get rid of the $\sqrt{2}$.  In 
homogeneous coordinates, the map is
$( z_0^2 - iz_1^2 , z_0^2 + iz_1^2 )$.  We are not done since we want to
change to the standard coordinates for the sphere.  So we must precompose
with a linear map that takes $J$ to
$\left[ \begin{smallmatrix} 1 & 0 \\ 0 & -1 \end{smallmatrix} \right]$.
We also multiply the first component by $i$ for simplicity.
In nonhomogeneous coordinates the map is
\begin{equation}
z \mapsto
\frac%
{z^2+2iz+1}%
{z^2-2iz+1} .
\end{equation}
We know this map cannot be equivalent to a monomial map because we used the
canonical blocks that were not diagonalizable.
It is not
hard to see that any polynomial degree-two map of the disc to the disc must
be $e^{i\theta}z^2$ for some $\theta \in \R$, so the map is also not equivalent to any polynomial map.
Other examples can be constructed in a similar way.
\end{remark}

Next let us classify the degree-two monomial maps.  The classification in
general follows from the work of D'Angelo~\cites{D:poly,D:prop1},
see also \cites{DLP,DL:complex,DL:families}.
First, the following lemma.

\begin{lemma} \label{monequiv}
Let $f$ and $g$ be monomial CR maps of spheres that take
$0$ to $0$,
\begin{equation}
f(z) = \bigoplus_\alpha a_\alpha z^\alpha , \quad
g(z) = \bigoplus_\alpha b_\alpha z^\alpha ,
\end{equation}
where each monomial is distinct.
Then $f$ and $g$ are spherically equivalent if and only if
there exists a permutation
$\sigma$ of
the variables taking $\alpha$ to $\sigma(\alpha)$ and
$\lvert a_\alpha \rvert = \lvert b_{\sigma(\alpha)} \rvert$.
\end{lemma}

The proof follows by results of \cite{D:prop1}; polynomial proper maps of
balls taking origin to origin are spherically equivalent
if and only if they are unitarily
equivalent.  That is, $f$ and $g$ are spherically equivalent if and only if
$f = g \circ U$ for a unitary matrix $U$.
By setting all but one of the
variables to zero, that is $z = (0,\ldots,0,1,0,\ldots,0)$,
with $z_k = 1$,
we find that $z$ is in the sphere and therefore
$f(z)$ is in the sphere.  Thus at least one component of
$f(z)$ must be nonzero, that is at least one of the monomials
must depend only $z_k$.  This was true for all $k$ and
therefore
there must be at least one pure monomial for each
variable.  The result follows by application of multinomial theorem.
We produce a slightly different proof for degree-two maps for convenience
and to illustrate the methods of this paper.

\begin{proof}[Proof for degree-two maps]
One direction is simple, if there is a permutation of the
variables such that
$\lvert a_\alpha \rvert = \lvert b_{\sigma(\alpha)} \rvert$,
then obviously 
$f$ and $g$ are spherically equivalent.

For the other direction, suppose that $f$ and $g$ are monomial, degree-two, 
and spherically equivalent.
We write down the forms corresponding to
$\norm{f(z)}^2 - 1$ and $\norm{g(z)}^2 - 1$ in homogeneous coordinates
and find the matrix $A_f$ for $f$ and $A_g$ for $g$.  We notice
that these matrices must be diagonal.  They are in canonical form,
and the canonical form is canonical up to permutation of the blocks (hence
permutation of variables).  Hence after a permutation of variables
we get that the diagonal entries of $A_f$ and $A_g$ are equal and hence
$\lvert a_\alpha \rvert^2 = \lvert b_{\sigma(\alpha)} \rvert^2$.
\end{proof}

Hence the maps in \eqref{alldeg2maps} are all spherically inequivalent.
We will give a simple proof that the list is exhaustive, but we also note
that this follows easily from the general theory of polynomial proper maps
developed by D'Angelo.
In
\cite{DAngelo:CR} it is shown that all degree-two polynomial maps
preserving the origin
are given in the form $Lz \oplus (\sqrt{I-L^*L}\, z) \otimes z$.  
By Lemma~\ref{lemma:mapsmonomial},
we only need to consider diagonal $L$.
See also
\cites{DKR,DLP} for more details on the classification of monomial
maps.  In general, all polynomial (and hence monomial) maps are obtained
by a finite series of partial ``tensorings'' and ``untensorings.''  Using that
language, it is also not hard to see that the list \eqref{alldeg2maps}
is exhaustive.  For simplicity, we will not use this language here.

We will give
an elementary
argument in the real-algebraic language.
In the form $\norm{f(z)}^2$ replace
$\abs{z_1}^2$ with $x_1$,
$\abs{z_2}^2$ with $x_2$, etc\ldots.
Hence for a monomial map $f$, the form $\norm{f(z)}^2$ will become
a real polynomial in $x_1,\ldots,x_n$ with nonnegative coefficients.
This polynomial gives the equivalence class of all monomial maps up to
postcomposing with a diagonal unitary matrix of course.
So we must
classify
all real polynomials $p(x)$ of degree 2 with nonnegative coefficients
such that $p(x)=1$ when $x_1+\cdots+x_n = 1$.

Without any loss of generality,
we allow adding zero components to the map.  We have defined
spherical equivalence of maps with different target dimensions.
The following lemmas (and the theorem)
could be given without this convenience, but would be more
complicated to state.

\begin{lemma}
Any monomial map of spheres is spherically equivalent to a monomial map
taking the origin to the origin by postcomposing with
a diagonal matrix.
\end{lemma}

\begin{proof}
Suppose that $p(x)$ is a polynomial
with nonnegative coefficients,
such that $p(x) = 1$ whenever 
$x_1+\cdots+x_n = 1$, and $p(0) \not= 0$.  Note that $p(0) < 1$ unless $p$ is
trivial.  Then $\frac{p(x)-p(0)}{1-p(0)}$ is a
polynomial with nonnegative coefficients with no constant term that is one
on $x_1+\cdots+x_n = 1$.  The equivalence of the induced maps is by a
diagonal matrix.
\end{proof}

\begin{lemma}
All degree-two monomial maps of spheres taking the origin to the origin
are spherically equivalent to a map of the form
\eqref{alldeg2maps} by composing with
permutation matrices and a unitary matrix.
\end{lemma}

\begin{proof}
Let us show how we construct all the degree-two monomial examples in
\eqref{alldeg2maps}.
Suppose that we have a
polynomial $p(x)$ in $(x_1,\ldots,x_n)$ of degree 2, such that
$p-1$ is divisible by
$(x_1+\cdots+x_n-1)$, $p(0) = 0$, and all coefficients of $p$ are nonnegative.
Write $p = p_1 + p_2$ where $p_1$ is of degree 1, and $p_2$ is of degree 2.
$p_2$ must be divisible by $(x_1+\cdots+x_n)$.  This fact is easy to
see by homogenizing to $p_1(x) t + p_2(x) - t^2$, noting that this polynomial
is divisible by 
$(x_1+\cdots+x_n-t)$ and then setting $t=0$.
We find $q = p_1 + p_2 / (x_1+\cdots+x_n)$.
The polynomial $q-1$ is also divisible by
$(x_1+\cdots+x_n-1)$.  As $q$ is of degree 1,
$q-1$ is a constant multiple of
$(x_1+\cdots+x_n-1)$.  As $q(0)=0$ we must have
$q = x_1+\cdots+x_n$.  So
$p$ was constructed from $q$ by a partial tensoring
operation, which in this language is simply multiplication of a certain
terms of $q$
by $(x_1+\cdots+x_n)$.

Therefore, the list in \eqref{alldeg2maps} is exhaustive, as all the monomial
maps in this list are obtained in the above fashion.
\end{proof}

\begin{remark}
We mentioned
that there exists a third degree rational CR map
of spheres not spherically equivalent to a polynomial one~\cite{FHJZ}.  Let
us give an alternative simple argument using same reasoning as above,
for a slightly weaker statement.
\end{remark}

\begin{prop} \label{exists3deg}
There exists a rational third degree CR rational map $f \colon S^{2n-1} \to S^{2N-1}$
($n \geq 2$), not
spherically equivalent to a polynomial map taking the origin to origin.
In fact, when $n=2$, there exists a real two-dimensional family of such maps.
\end{prop}

\begin{proof}
First we note what the Hermitian form for a polynomial map taking origin to
origin looks like.  That is, the matrix $B$ in
$\langle B \sZ_3 , \sZ_3 \rangle$, where $\sZ_3$ is the degree-three Veronese
map.
It is not hard to see that $B = B_1 \oplus 1$ for some matrix $B_1$
with one less row and one less column than $B$
and the $1 \times 1$ matrix $1$.  Therefore, let $\sZ_2$
be the degree-two Veronese map and $A$ is the matrix such that
$\langle B \sZ_3 , \sZ_3 \rangle
= \langle A \sZ_2, \sZ_2 \rangle \langle J z , z \rangle$.  Then
$A = A_1 \oplus 1$ (as for $B$).
Suppose that $n=2$ for simplicity, so $A$ is $6 \times 6$.

By results of D'Angelo (see~\cite{DL:families} Proposition 3, for example, for explicit
statement) a polynomial map of spheres of degree $d$ is determined by its
$d-1$ jet (the coefficients of all monomials
of degree $d-1$ or less).
That is, no two distinct polynomial CR maps of spheres
have the same $d-1$ jet.  Furthermore, there exists
an open neighbourhood of the origin in the $d-1$ jet space
such that each such jet gives a CR map of spheres.
Therefore, there
exists an open set of $A$ in the space of
$6 \times 6$ Hermitian matrices that correspond to
degree-three rational maps of spheres.
For the map to be equivalent to a polynomial map
taking the origin to origin, we must set $5$ complex parameters in $A$
to zero.  That is, for $A$ to correspond to a polynomial map preserving
the origin, then it must be of the form $A_1 \oplus 1$ where $A_1$ is
a $5 \times 5$ Hermitian matrix.
The
set of linear maps of the source preserving $J$ has complex dimension 3.
Hence, since an open set of the $d-1$ jets is possible, then we will not
always be able to set all $5$ parameters to zero.  In fact there will be a whole
family of examples.  In nonhomogeneous coordinates we see that there must be
a real two-dimensional family of examples of maps not spherically equivalent to
a polynomial map taking the origin to origin.
\end{proof}


\section{Monomial maps of hyperquadrics} \label{section:monomial}

Let us study the monomial version of the problem of classifying monomial
maps between hyperquadrics in dimensions 2 and 3.
The proofs in this simplified case
illustrate the combinatorics of the more general problem.  Furthermore, we
also use the monomial classification for the general case.
It appears that the combinatorics governing the monomial situation govern the
general situation of CR maps of hyperquadrics in some sense.
For example, even allowing for
higher source and target dimension, there is no CR map of spheres known to the
author that is not homotopic to a monomial example.  We have also shown
in this
paper that all degree-two maps of spheres are equivalent to monomial maps.

The setup translates into a simpler problem in real-algebraic
geometry as it did for monomial CR map of spheres.
On the other hand, due to a different equivalence relation, we get a somewhat
more complicated statement.
Here, $Q(1,1)$ is no longer equivalent
to $Q(2,0)$, and we furthermore have a perfectly valid interpretation of
$Q(0,2)$.

See \cites{DKR, DLP} for more on the following setup.
Let 
$f \colon Q(a,b) \to Q(c,d)$
be a monomial CR map.
We 
have that $\abs{f_1}^2 + \cdots + \abs{f_c}^2 -
\abs{f_{c+1}}^2 - \cdots - \abs{f_{c+d}}^2 = 1$ when
$\abs{z}^2 + \epsilon \abs{w}^2 = 1$, where $\epsilon = 1$ if
the source is $Q(2,0)$ and $-1$ when the source is $Q(1,1)$.  As every
$f_j$ is a monomial, we can replace $x = \abs{z}^2$ and $y = \abs{w}^2$.

We can therefore formulate the following real-algebraic problem.  Let
$p(x,y)$ be a real polynomial that has $N = N(p) = N_+(p) + N_-(p)$
distinct monomials,
$N_+$ monomials with a positive coefficient and
$N_-$ monomials with a negative coefficient.  We will call
the tuple $(N_+(p),N_-(p))$ the \emph{signature} of $p$.
We will consider two polynomials $p$ and $q$ equivalent if either
$p(x,y) = q(x,y)$ or $p(x,y) = q(y,x)$.
Then we have the following classification.  Note that
in the real-algebraic problem signature $(0,k)$ makes sense.

We will state the general version in homogeneous coordinates.  First, let us 
give the specific result for spheres in nonhomogeneous coordinates.

\begin{prop}
Let $p(x,y)$ be a real polynomial such that $p(x,y) = 1$
whenever $x+y=1$.  Then the possible polynomials (up to swapping of
variables) with signature $(3,0)$ are
\begin{enumerate}[(i)]
\item $x^3 + 3xy + y^3$,
\item $x^2 + 2xy + y^2$,
\item $x + xy + y^2$,
\item $\alpha x + \alpha y + (1-\alpha)$ for some $\alpha \in (0,1)$.
\end{enumerate}
For signature $(2,1)$ we have
\begin{enumerate}[(i)]
\item $x^2 + 2y - y^2$,
\item $\alpha x + \alpha y - (\alpha-1)$ for some $\alpha > 1$.
\end{enumerate}
For signature $(1,2)$ we have
\begin{enumerate}[(i)]
\item $(1+\alpha) - \alpha x - \alpha y$ for some $\alpha > 0$.
\end{enumerate}
No such polynomial with signature $(0,3)$ exists.
\end{prop}

Instead of directly proving this particular case, we give a more general
statement.  Suppose that we have a real homogeneous polynomial $p(x,y,t)$,
such that $p=0$ on $x+y+t = 0$.  If we classify all such polynomials with 4
(or less) terms, then we also prove the proposition.  We take
$p(x,y,-t)$, set $t=1$, and only take those polynomials with a constant term
equal to one.

In this way we also easily obtain all monomial maps from $Q(1,1)$ to 
other hyperquadrics.
Therefore, we need only to prove the following lemma.

\begin{lemma} \label{monhqclass}
Let $p(x,y,t)$ be a nonzero
homogeneous polynomial with four or less distinct
monomials such that $p(x,y,t) = 0$ on $x+y+t = 0$.  Then, up to permutation of
variables, we have $p = qm$, where $m$ is an arbitrary monomial and $q$ is one
of the following polynomials
\begin{enumerate}[(i)]
\item $x^3+y^3+t^3-3xyt$,
\item $x^2+y^2+2xy-t^2$,
\item $x^2+xy-yt-t^2$,
\item $x+y+t$.
\end{enumerate}
\end{lemma}

Before we get to the proof, let us prove the following proposition.  It is
a modification of the classification of homogeneous monomial CR map of
spheres.
See \cites{DKR,DLP} for the sphere version and a slightly different proof.

\begin{prop} \label{homomonprop}
Let $p(x,y,t) = \varphi(x,y) - \alpha t^d$ for some real number $\alpha$
be a homogeneous polynomial
such that $p=0$ when $x+y+t=0$.  Then $\varphi(x,y) = \alpha (-1)^d(x+y)^d$.
\end{prop}

\begin{proof}
Note that $-(x+y)=t$ on $x+y+t=0$.  Hence write
\begin{equation}
\varphi(x,y) - \alpha (-1)^d(x+y)^d .
\end{equation}
This is a polynomial that is zero on $x+y+t=0$, but does not depend on $t$
and hence must be identically zero.
\end{proof}

\begin{proof}[Proof of Lemma~\ref{monhqclass}]
By dividing through by $m$ we can assume that the monomials in $p$ have
no common factor.
Next we notice that $p(x,y,t)$ is divisible by $x+y+t$.  First suppose that
$p(x,y,t)$ has 3 or fewer monomials.  For each variable,
there must be a monomial containing it by our assumption.  Hence
for example $p(x,y,0)$ has at most 2 monomials, but also cannot be
identically zero by assumption.  Since it is divisible by $x+y$
we can see that $p(x,y,0) = m(x,y)(x^k\pm y^k)$, for some monomial $m(x,y)$
depending on $x$ and $y$ only.  This fact is easy to see,
and can be proved in similar fashion as Proposition~\ref{homomonprop}.
Same reasoning applies for the other variables.  As there can be at most $3$
monomials
then $p$ must be a constant multiple of $x^k+y^k+t^k$.  Now by 
application of Proposition~\ref{homomonprop} $k=1$.

Therefore we can assume that there are exactly 4 distinct monomials.
It is elementary to see that at least two variables divide
two or more terms in $p$.  Without loss of generality, suppose these are
the variables $x$ and $y$.  By the same logic as above,
we have $p(x,0,t) = m_1(x,t)(x^k\pm t^k)$ and
$p(0,y,t) = m_2(y,t)(y^\ell\pm t^\ell)$ for
some monomials $m_1$ and $m_2$.  As $t$ cannot divide
all terms, and as we have exactly 4 terms, we conclude, without loss
of generality, that $m_1 = \mu x^{d-k}$.  We can scale $p$ to have $\mu = 1$.
We have two possibilities.  The polynomial $p$ is either
\begin{equation}
x^d \pm x^{d-k} t^k + \alpha y^{a+\ell} t^b \pm \alpha y^a t^{b+\ell}
\qquad
\text{or}
\qquad
x^d \pm t^d \pm y^\ell t^{d-\ell} + \alpha x^ay^bt^c .
\end{equation}
To discount the first possibility note that $k \geq 1$ and $\ell \geq 1$.
If we let $t=0$ we must get at
least two terms as before, so $b=0$.  We also note that this must mean
that $\alpha = \pm 1$ (again by setting $t=0$).  We are left with
essentially a special case of the second possibility, so let us focus on
that.

First suppose that $\ell = d$.  Then we have
$x^d \pm t^d \pm y^d + \alpha x^ay^bt^c$.  If $c = 0$, then we apply
Proposition~\ref{homomonprop} to conclude that $d=2$ and the polynomial
is $x^2 - t^2 + y^2 + 2 xy$.  Similar with $a$ and $b$.

Now let us replace $t$ in the 
last term with $(-x-y)$.  We get
$x^d \pm t^d \pm y^d + \alpha x^ay^b(-x-y)^c$.  This polynomial fulfills the
hypothesis of Proposition~\ref{homomonprop}.  Hence 
$x^d \pm y^d + \alpha x^ay^b(-x-y)^c$ must have the same number of terms as
$(x+y)^d$.  We conclude that $c=d-2$.  Similarly, we repeat the argument for
$x$ and $y$ and conclude that $a=d-2$ and $b=d-2$.  Hence $d=3$.  By dividing
by $x+y+t$ we see that the polynomial must be
$x^3+t^3+y^3-3xyt$.

So suppose that $\ell < d$.  Then after setting $t =0$ we must have two
terms and so $c=0$.  By setting $t=0$ we also see that $\alpha=\pm1$.  So
we have
$x^d \pm t^d \pm y^\ell t^{d-\ell} \pm x^ay^b$.
We replace $x$ with $(-y-t)$ and apply Proposition~\ref{homomonprop}.
After counting monomials we find that $a=d-2$ or $a=d-1$.
Letting $t=(-x-y)$ and doing the same as above we find that $d-\ell = d-2$
or $d-\ell = d-1$.  Hence $y$ appears to at most second power.
If we substitute $y=-x-t$ we must get an identically zero polynomial.
By counting monomials we get a finite list of possibilities and the only one
that is divisible by $x+y+t$ is 
$x^2-t^2+yt-xy$.
\end{proof}

We can now obtain monomial CR maps of hyperquadrics by possibly permuting
variables and substituting
$\pm\abs{z}^2$ for $x$, $\pm\abs{w}^2$ for $y$ and $\pm 1$ for $t$.
We can also obtain CR maps of hyperquadrics by substituting
$\pm z\bar{w}$ for $x$, $\pm w\bar{z}$ for $y$ and $\pm 1$ for $t$ and
then changing variables to put the source hyperquadric
into the standard form.  We can follow this procedure 
to find the degree-3 CR map of $Q(2,0)$ to $Q(2,1)$.  Note that we
will never get a CR map of spheres by this nonstandard substitution.


\section{Degree 2 maps from \texorpdfstring{$S^3$ to $Q(2,1)$}{S\textthreesuperior\ to Q(2,1)}} \label{section:deg2mapshq23}

Let
$f \colon S^3 = Q(2,0) \to Q(2,1)$ be
a degree-two rational
map.  That is, let $J =
\left[
\begin{smallmatrix}
1 & 0 & 0 \\
0 & 1 & 0 \\
0 & 0 & -1
\end{smallmatrix}
\right] = \operatorname{diag}(1,1,-1)$, and let $A$ be a Hermitian $3 \times
3$ matrix such that
in homogeneous coordinates (using the notation above)
\begin{equation}
\langle V f(z) , f(z) \rangle =
\langle F^* V F \sZ , \sZ \rangle =
\langle A z, z \rangle 
\langle J z, z \rangle ,
\end{equation}
where $V = \operatorname{diag}(1,1,-1,-1)$.  I.e.\ the $6 \times 6$
matrix $B = F^* V F$ is rank 4 or less.
As the sphere does not contain any
complex varieties it is easy to see that the rank cannot be less than 3.
If $B$ is rank 3, then the map would be CR map of
spheres in the same dimension and hence Q-equivalent to the identity map,
which is not degree-two.  Thus assume that $B$ is of rank 4,
with 2 positive and 2 negative eigenvalues.

As explained above we 
use the automorphism group of the ball to put $A$ into a canonical form.
We have already computed a list of all possible canonical forms and the
resulting matrices $B$.  All we have to do is find those canonical
forms for which $B$ is rank 4 and has 2 positive and 2 negative eigenvalues.

If we can diagonalize the pair $(J,A)$
simultaneously by $*$-congruence, the map is equivalent to a monomial map,
that is, monomial in homogeneous coordinates.
As mentioned before, we know we have a monomial map from
some hyperquadric equivalent to the ball.  Hence by the classification
of monomial CR maps of hyperquadrics, we are essentially done.
Applying
Lemma~\ref{monhqclass}
we find the following list of monomial CR maps of hyperquadrics
in homogeneous coordinates $(z,w,t)$
\begin{align}
& (z,w,t) \mapsto (z^2,\sqrt{2}\,wt,w^2,t^2), \label{firsthqmonmap} \\
& (z,w,t) \mapsto (tz, w^2, wt, z^2). \label{secondhqmonmap}
\end{align}
The
matrix $A$ for these two maps is different even after permutation or
negation (we need to handle negation,
since $V$ has two positive and two negative eigenvalues). 
Consequently, the two maps \eqref{firsthqmonmap} and
\eqref{secondhqmonmap} are not Q-equivalent.

We focus on the case where we cannot diagonalize.  Let us again 
assume that there is some block in the canonical form greater than $1 \times
1$.
We will consider all the computed canonical forms for $B$ for
all the matrix pairs $(J,A)$ from \S~\ref{section:deg2maps}.

Let us take the situation corresponding to
the block $\Delta_2$ in $J$.  The canonical pair for $(J,A)$ is
then
$\left(
\left[
\begin{smallmatrix}
       1 &        0 &        0\\
       0 &        0 &        1\\
       0 &        1 &        0
\end{smallmatrix}
\right]
,
\left[
\begin{smallmatrix}
  \alpha &        0 &        0\\
       0 &        0 &    \beta\\
       0 &    \beta &        1
\end{smallmatrix}
\right]
\right)$.
The corresponding $B$ is computed in \eqref{form1B}.
This matrix can never have rank 4.  It can have rank 6, 5, or 3.  Therefore
we do not have to consider this case.
Similar calculations show the same result for the related canonical
form
$\left(
\left[
\begin{smallmatrix}
       1 &        0 &        0\\
       0 &        0 &        -1\\
       0 &        -1 &        0
\end{smallmatrix}
\right]
,
\left[
\begin{smallmatrix}
  \alpha &        0 &        0\\
       0 &        0 &    \beta\\
       0 &    \beta &        -1
\end{smallmatrix}
\right]
\right)$.

Next let us consider the pair
$\left(
\left[
\begin{smallmatrix}
       0 &        0 &        1\\
       0 &        1 &        0\\
       1 &        0 &        0
\end{smallmatrix}
\right]
,
\left[
\begin{smallmatrix}
       0 &        0 &   \alpha\\
       0 &   \alpha &        1\\
  \alpha &        1 &        0
\end{smallmatrix}
\right]
\right)$.
The corresponding $B$ is computed in \eqref{form2B}.
This matrix is rank 4 if and only if $\alpha=0$.  The matrix has
two positive and two negative eigenvalues, and so we do get a map to the
$Q(2,1)$ hyperquadric.

We follow the procedure to get a map as follows. First change variables to
put $J$ into the form 
$
\left[
\begin{smallmatrix}
       1 &        0 &        0\\
       0 &        1 &        0\\
       0 &        0 &        -1
\end{smallmatrix}
\right]$.  We then find a set of orthonormal eigenvectors and follow the
procedure outlined before to get the map in homogeneous coordinates
\begin{multline}
(z,w,t) \mapsto (
z^2+\sqrt{3}\,zw+w^2-zt ,
w^2+zt-\sqrt{3}\,wt - t^2 , \\
z^2-\sqrt{3}\,zw+w^2-zt,
w^2+zt+\sqrt{3}\,wt - t^2) .
\end{multline}
Note that we have scaled the eigenvectors to avoid ugly expressions.

Finally we get to the canonical pair
$\left(
\left[
\begin{smallmatrix}
       1 &        0 &        0\\
       0 &        0 &        1\\
       0 &        1 &        0
\end{smallmatrix}
\right]
,
\left[
\begin{smallmatrix}
  \alpha &        0 &        0\\
       0 &        0 & \beta-i\gamma\\
       0 & \beta+i\gamma &        0
\end{smallmatrix}
\right]
\right)$,
for $\gamma > 0$.
The corresponding $B$ is computed in \eqref{form3B}.
The only way for the matrix $B$ to have rank 4 is to set $\alpha = \beta =
0$.  We can rescale to have $\gamma=1$.  We follow the same procedure as
before, we change variables and find a set of orthonormal eigenvectors
as before to obtain the map
\begin{equation}
(z,w,t) \mapsto
(\sqrt[4]{2}(zw-izt),
w^2-\sqrt{2}\,iwt+t^2,
\sqrt[4]{2}(zw+izt),
w^2+\sqrt{2}\,iwt+t^2).
\end{equation}

We are done.  We have proved the following lemma.

\begin{lemma} \label{lemma:deg2hq23}
Let $f \colon Q(2,0) \mapsto Q(2,1)$ be a rational degree-two CR map.
Then
$f$ is Q-equivalent to exactly one of the maps
\begin{enumerate}[(i)]
\item
$(z,w) \mapsto (z^2,\sqrt{2}\,w,w^2)$,
\item
$(z,w) \mapsto \left( \frac{1}{z}, \frac{w^2}{z^2}, \frac{w}{z^2} \right)$,
\item
$(z,w) \mapsto \left(
\frac{z^2+\sqrt{3}\,zw+w^2-z}{w^2+z+\sqrt{3}\,w - 1} ,
\frac{w^2+z-\sqrt{3}\,w - 1}{w^2+z+\sqrt{3}\,w - 1} ,
\frac{z^2-\sqrt{3}\,zw+w^2-z}{w^2+z+\sqrt{3}\,w - 1}
\right)$,
\item
$(z,w) \mapsto \left(
\frac{\sqrt[4]{2}(zw-iz)}{w^2+\sqrt{2}\,iw+1} ,
\frac{w^2-\sqrt{2}\,iw+1}{w^2+\sqrt{2}\,iw+1} ,
\frac{\sqrt[4]{2}(zw+iz)}{w^2+\sqrt{2}\,iw+1}
\right)$.
\end{enumerate}
\end{lemma}


\section{Proof of Theorem \ref{thm:hqclass}} \label{section:proofofhqclass}

To prove the classification of CR maps of hyperquadrics, we will have to
revisit some of Faran's~\cite{Faran:B2B3} setup.

Suppose that a holomorphic function $f \colon U \subset \C^2 \to \C^3$
is such that an image of any complex line is contained in some plane.
We then call $f$ a \emph{planar} map.
If $f$ is a planar map defined on a domain $U \subset \bP^2$,
we can define the \emph{dual} map.
First for a domain $U$ define the dual domain $U^*$
where $U^*$ is composed of lines $L \subset \bP^2$ such that
$f(L \cap U)$ is contained in a unique plane in $\bP^3$.
Let $\bP^{2*}$ and $\bP^{3*}$ be the dual spaces, that is
the spaces of hyperplanes in $\bP^2$ and $\bP^3$ respectively.  We
can then define
$f^* \colon U^* \subset \bP^{2*} \to \bP^{3*}$ so that if $f$
takes a line $L$ into the unique plane $P$, then $f^*(L) = P$.
See~\cite{Faran:B2B3} for basic properties of duals.  For example,
the dual of a dual is the map itself.
Let $f$ and $g$ be two maps of $\bP^2$ to $\bP^3$.
If there exist conjugate-linear
isomorphisms
$\psi$ and $\varphi$ such that $\psi f = g \varphi$, then
$f$ and $g$ are said to be conjugate isomorphic.  If $f$ and
$f^*$ are conjugate isomorphic, then $f$ is said to be self-dual.

The following lemma gives a useful property of CR maps of hyperquadrics.
Faran proved the lemma for CR maps of sphere.
The proof for hyperquadrics is
almost exactly the same and we restate it here for convenience.

\begin{lemma}[Faran~\cite{Faran:B2B3}] \label{lemma:selfdual}
If $f \colon U \subset Q(2,0) \to Q(2,1)$,
is a nonconstant real-analytic CR map, then $f$ is planar and self-dual.
\end{lemma}

\begin{proof}
Let $f \colon U \subset Q(2,0) \to Q(2,1)$ be a nonconstant real-analytic
CR map.  As $f$ is real-analytic, it extends to a holomorphic map of
a neighborhood of $U$ in $\C^2$ to $\C^3$.  Let us call this extension $f$
for simplicity.  As $f$ maps $Q(2,0)$ to $Q(2,1)$ then in homogeneous
coordinates we have
$\langle V f(z) , f(z) \rangle = 0$
whenever
$\langle J z , z \rangle = 0$, where $V$ defines $Q(2,1)$ and
$J$ defines $Q(2,0)$.  Then by polarization we note that the line
defined by
$\langle J z , w \rangle = 0$ for a fixed point $w$, gets mapped to the plane
defined by
$\langle V \zeta , f(w) \rangle = 0$.  Hence $f$ is a planar map.

Let $\xi_0,\xi_1,\xi_2$ be the homogeneous coordinates for $\bP^2$
and let the image by $f$ be the point
$\xi_0',\xi_1',\xi_2',\xi_3'$ in homogeneous coordinates for $\bP^3$.
Suppose that our nonhomogeneous coordinates correspond to
$z_1 = \frac{\xi_1}{\xi_0}$, 
$z_2 = \frac{\xi_2}{\xi_0}$, and for $\bP^3$
we have
$z_1' = \frac{\xi_1'}{\xi_0'}$, 
$z_2' = \frac{\xi_2'}{\xi_0'}$,
$z_3' = \frac{\xi_3'}{\xi_0'}$.

We let $w_1 = \frac{\eta_1}{\eta_0}$, etc\ldots  
By polarization
the statement
that $f$ takes $Q(2,0)$ to $Q(2,1)$ becomes:
\begin{equation} \label{xieta1}
\xi_0 \bar{\eta}_0 - \xi_1 \bar{\eta}_1 - \xi_2 \bar{\eta}_2 = 0
\end{equation}
implies
\begin{equation} \label{xieta2}
\xi_0' \bar{\eta}_0' + \xi_1' \bar{\eta}_1'
- \xi_2' \bar{\eta}_2' - \xi_3' \bar{\eta}_3' = 0 .
\end{equation}
The equation \eqref{xieta1} defines a line.  In dual coordinates
the line is $(\bar{\eta}_0,{-\bar{\eta}_1},{-\bar{\eta}_2})$.  Similarly
\eqref{xieta2} defines a plane, in dual coordinates
$(\bar{\eta}_0',\bar{\eta}_1',{-\bar{\eta}_2'},{-\bar{\eta}_3'})$.
$f^*$ takes
$(\bar{\eta}_0,{-\bar{\eta}_1},{-\bar{\eta}_2})$ to
$(\bar{\eta}_0',\bar{\eta}_1',{-\bar{\eta}_2'},{-\bar{\eta}_3'})$.
Hence we have defined the conjugate isomorphisms 
$\psi$ and $\varphi$ such that $\psi f^* = f \varphi$.  That is,
$f$ is self-dual.
\end{proof}

Faran proved the following complete classification of planar maps.
The designations of non-degenerate, partially degenerate, developable, degenerate
and flat maps are not relevant to us, but are kept for consistency.
We say two maps $f,g \colon U \subset \bP^2 \to \bP^3$ are equivalent
if there exist $\tau \in \operatorname{Aut}(\bP^3)$ and
$\chi \in \operatorname{Aut}(\bP^2)$
if $f \circ \chi = \tau \circ g$.

\begin{thm}[Faran~\cite{Faran:B2B3}] \label{thm:planar}
Let $f \colon U \subset \bP^2 \to \bP^3$ be a planar immersion.  Then
$f$ is equivalent to one of the following maps.

\begin{enumerate}[A.]
\item Non-degenerate maps.
\begin{enumerate}[1.]
\item $(z_0, z_1, z_2) \mapsto (z_0^3, z_1^3, z_2^3, z_0 z_1 z_2)$,
\item
\begin{enumerate}[a)]
\item $(z_0, z_1, z_2) \mapsto
(z_0^2 + \frac{3}{4} z_1^2 + \frac{3}{4} z_2^2,
 z_0 z_1 + \frac{1}{2} z_2^2,
 z_0 z_2 + \frac{1}{2} z_1^2,
 z_1 z_2)$,
\item $(z_0, z_1, z_2) \mapsto
(z_0^2 - \frac{15}{16} z_1^2 - \frac{3}{4} z_2^2,
 z_0 z_1 + \frac{1}{2} z_2^2,
 z_0 z_2 + \frac{1}{2} z_1^2,
 z_1 z_2)$,
\item $(z_0, z_1, z_2) \mapsto
(z_0^2,
 z_0 z_1 + \frac{1}{2} z_2^2,
 z_0 z_2 + \frac{1}{2} z_1^2,
 z_1 z_2)$.
\end{enumerate}
\end{enumerate}
\item Partially degenerate maps.
\begin{enumerate}[1.]
\item
\begin{enumerate}[a)]
\item Dual to A.2.a.
\item Dual to A.2.b.
\item Dual to A.2.c.
\end{enumerate}
\item $(z_0, z_1, z_2) \mapsto (z_0^2, z_1^2, z_0 z_2, z_1 z_2)$,
\item $(z_0, z_1, z_2) \mapsto
(z_0^2, z_0 z_1, z_0 z_2 + \frac{1}{2} z_1^2, z_1 z_2)$,
\end{enumerate}
\item Developable map
$(z_0, z_1, z_2) \mapsto (z_0^2, z_1^2, z_1 z_2, z_2^2)$.
\item Degenerate maps
$(z_0, z_1, z_2) \mapsto (z_0, z_1, z_2, z_0 g(z_1/z_0, z_2/z_0))$
for any function $g$
\item Flat maps
$(z_0, z_1, z_2) \mapsto (1, g_1(z_1/z_0, z_2/z_0), g_2 (z_1/z_0, z_2/z_0), 0)$
for any functions $g_1$ and $g_2$.
\end{enumerate}
Except for degenerate and flat maps, no one of the above maps is equivalent
to another.
\end{thm}

Two CR maps of hyperquadrics could be equivalent, but not Q-equivalent.
Hence, to use the classification of planar maps for the CR hyperquadric maps, we
must first check if the class of planar maps contains a CR
hyperquadric map, and
second if the class contains other non-Q-equivalent CR hyperquadric maps.

We can apply our result for degree-two maps to handle all the degree-two
cases.  Hence we only need to study cases A.1, D, and E\@.  Case A.1 is handled
by the following Lemma.

\begin{lemma} \label{lemma:deg3}
Suppose that a homogeneous polynomial map of $\bP^2 \to \bP^3$ induced by
$f \colon Q(2,0) \mapsto Q(2,1)$ is equivalent to
$(z_0, z_1, z_2) \mapsto (z_0^3, z_1^3, z_2^3, z_0 z_1 z_2)$.  Then
$f$ is Q-equivalent to
$(z,w) \mapsto \left( \frac{2w^3}{3z^2+1} ,
\frac{z^3+3z}{3z^2+1},
\sqrt{3}\frac{wz^2-w}{3z^2+1}
\right)$.
\end{lemma}

\begin{proof}
Faran computes that $f$ is self-dual.  After changing coordinates
the map
\begin{equation}
(z_0, z_1, z_2) \mapsto (z_0^3, z_1^3, z_2^3, \sqrt{3} \, z_0 z_1 z_2)
\end{equation}
takes sphere to the sphere.
By Lemma~\ref{lemma:selfdual}, $f$ is self-dual and it is not hard to compute
that the dual of $f$ is
$(\eta_0, \eta_1, \eta_2) \mapsto (\eta_0^3, \eta_1^3, \eta_2^3, - \sqrt{3}
\, \eta_0 \eta_1 \eta_2)$.

If we have conjugate isomorphisms
$\psi_n \colon \bP^n \to \bP^{n*}$ such that
$f^* \psi_2 = \psi_3 f$, then
$f$ takes the set $z \cdot \psi_2(z) = 0$ into
the set
$w \cdot \psi_3(w) = 0$.
These zero sets
are hyperquadrics if the corresponding forms are Hermitian.  We then check
the signature of the forms.
Thus to find all ways that $f$ maps a hyperquadric to a hyperquadric,
we need only find all the ways that $f$ is conjugate isomorphic to $f^*$.

Above we found
\begin{equation}
\psi_2(z) = (-\bar{z}_0,\bar{z}_1,\bar{z}_2), \qquad
\psi_3(w) = (-\bar{z}_0,\bar{z}_1,\bar{z}_2,\bar{z}_3),
\end{equation}
which give the CR map of spheres.

Suppose that
$(\psi_2,\psi_3)$ and
$(\psi_2',\psi_3')$
are conjugate isomorphisms of $f$ and $f^*$
that give Hermitian forms (that is, give maps of hyperquadrics).
We have $f^* = \psi_3 f \psi_2^{-1} = \psi_3' f \psi_2'^{-1}$, hence
$f \psi_2^{-1} \psi_2' = \psi_3^{-1} \psi_3' f$.  Therefore
the pair
\begin{equation}
(\alpha_2, \alpha_3) =
(\psi_2^{-1} \psi_2', \psi_3^{-1} \psi_3')
\end{equation}
is a pair of automorphisms of the source and target that fix $f$, that is
$f \alpha_2 = \alpha_3 f$.  We 
denote by $\operatorname{Aut}(f)$ the
set of such pairs
$(\alpha_2, \alpha_3)$.  We have found $(\psi_2,\psi_3)$ above,
hence by
computing $\operatorname{Aut}(f)$ we find all hyperquadrics that $f$
takes to other hyperquadrics.  We need only check they are of the right
signature and that we do not get Q-equivalent maps.

Now we note that the Jacobian of $f$ drops rank at precisely the coordinates
$(1,0,0)$, $(0,1,0)$, $(0,0,1)$ in $\bP^2$.  So any automorphism preserves
these points.  Any permutation of the source variables
$(z_0,z_1,z_2)$ is part of an automorphism of $f$.  Hence any automorphism
of $f$ is a composition of a permutation of the variables and the map
\begin{equation}
(z_0,z_1,z_2 \, ; \, 
w_0,w_1,w_2,w_3)
\mapsto
(az_0,bz_1,cz_2 \, ; \, 
a^3w_0,b^3w_1,c^3w_2,abcw_3) .
\end{equation}
Here the first three components (in the $z_j$ variables)
of the map represent the self map of
$\bP^2$ and the last four components (in the $w_k$ variables) represent
the self map of $\bP^3$.

Now we need to check all the permutations of variables to see what maps
we get by considering different elements of
$\operatorname{Aut}(f)$, which has 6 components (corresponding to the
permutations).  However, due to symmetry in the map, we find that we need only
check the identity, the permutation of two variables, and the cyclic
permutation.

Faran computes that the identity component gives only maps Q-equivalent to
CR map of spheres.  It is also possible to use the computation below to show the
same result.  The cyclic permutation does not give Hermitian forms and hence
does not give CR maps to hyperquadrics.

Hence we need only check the permutation of two variables.  For example,
\begin{equation}
(z_0,z_1,z_2 \, ; \, 
w_0,w_1,w_2,w_3)
\mapsto
(az_2,bz_1,cz_0 \, ; \, 
a^3w_2,b^3w_1,c^3w_0,abcw_3) .
\end{equation}
The $\psi_2'$ and $\psi_3'$ now give the forms
\begin{align}
& z \cdot \psi_2'(z) = -\bar{a} z_0\bar{z}_2+\bar{b}z_1\bar{z}_1
+ \bar{c} z_2 \bar{z}_0, \\
& w \cdot \psi_3'(w) = -\bar{a}^3 w_0\bar{w}_2+\bar{b}^3w_1\bar{w}_1
+ \bar{c}^3 w_2 \bar{w}_0 + \bar{a} \bar{b} \bar{c} w_3 \bar{w}_3.
\end{align}
These forms are Hermitian if $b \in \R$ and $c = -\bar{a}$.  The first form
has two positive eigenvalues only if $b > 0$.  We can now change variables
$z \mapsto (c^{-1}z_0, b^{-1/2}z_1, z_2)$
to eliminate the variables $a,b,c$ from
the problem.  That is we have
\begin{align}
& z \cdot \psi_2'(z) = z_0\bar{z}_2+z_1\bar{z}_1 + z_2 \bar{z}_0, \\
& w \cdot \psi_3'(w) = w_0\bar{w}_2+z_1\bar{w}_1 + w_2 \bar{w}_0 - w_3 \bar{w}_3.
\end{align}
The forms are of correct signature and hence we have a map of $Q(2,0)$ to
$Q(2,1)$.

As a side note,
we get the
same map that we obtain by looking at the real-algebraic
version of the monomial problem.  Take the negative of
the homogenized version of the
degree-3 monomial map.  We have that
\begin{equation}
p(x,y,t) = x^3+ y^2 + t^3 - 3 xyt \quad \text{ is zero on }
\quad x+y+t = 0 .
\end{equation}
Now we compute $p(z_0\bar{z}_2, z_1\bar{z}_1, z_2\bar{z}_0)$ and this will
give the form
$w \cdot \psi_3'(w)$ after removing the Veronese map.  Hence, even this
non-monomial map comes about from a monomial map.

After changing variables we get the map (in nonhomogeneous coordinates)
\begin{equation} \label{deg3hq23}
(z,w) \mapsto \left( \frac{2w^3}{3z^2+1} ,
\frac{z^3+3z}{3z^2+1},
\sqrt{3}\frac{wz^2-w}{3z^2+1}
\right).
\end{equation}

We found no other maps and hence all degree-3 maps of $Q(2,0)$ to $Q(2,1)$
are Q-equivalent to the map \eqref{deg3hq23}.
\end{proof}

\begin{proof}[Proof of Theorem~\ref{thm:hqclass}]
Let $f \colon U \subset Q(2,0) \to Q(2,1)$ be a nonconstant real-analytic
CR map.  As $f$ is real-analytic, it extends to a holomorphic map of
a neighborhood of $U$ in $\C^2$ to $\C^3$.  Let us call this extension $f$
for simplicity.  The map $f$ is planar by Lemma~\ref{lemma:selfdual}.

First suppose that
the derivative
of the map is at most 1.  Then the derivative must be of rank 1 at generic
points of a neighbourhood of $U$.  That is, at a generic point,
$f$ takes a germ of $\C^2$ into a germ of a one-dimensional subvariety of
$\C^3$.  This behaviour occurs on the complement of a complex analytic set
in $\C^2$.  Therefore, there exists a point $p \in U \subset Q(2,0)$
such that $f$ takes a
neighborhood of $p$
in $\C^2$ into a one-dimensional complex subvariety in $\C^3$.
It is easy to
see that $f$ takes a neighborhood $W$ of $p$ in $\C^2$ into $Q(2,1)$.  If 
$f(W) \cap Q(2,1)$ would be of less than
two real dimensions, then we could pull back a point
and get a complex subvariety of $Q(2,0)$, which is impossible.

Therefore, $f$ takes a neighborhood of $p$ into an irreducible
complex subvariety contained in $Q(2,1)$.  After a change of coordinates on
the target side
we can assume that $f(p) = (1,0,0)$
and the complex variety inside $Q(2,1)$ that contains the image of $f$
is given by $\{ z \in \C^3 \mid z_2 = z_3, z_1 = 1 \}$.
We immediately get that $f$ is Q-equivalent to the map
$(z,w) \mapsto (1,g,g)$ for some CR function $g$.

So suppose that at generic points the rank of the derivative is 2.  Again as
as the rank can drop only on a complex variety, that means the rank of the
derivative must be 2 on an open and dense subset of $U \subset Q(2,0)$.
Taking a perhaps smaller neighbourhood $U$ we can simply assume that
the rank of the derivative is identically 2 and we can
assume that $f$ is an immersion.
We know that $f$ is equivalent to one of the maps in
Theorem~\ref{thm:planar}.  However, as we said before,
the type of equivalence is not quite correct (it is not Q-equivalence).
We need to check each class if it contains a CR map of hyperquadrics.  If
it does, we need to check if it contains several non-Q-equivalent maps.

If $f$ is of type E, then it must map into a plane $P$.  $Q(2,1) \cap P$ 
is equivalent to $Q(2,0)$ or $S^1 \times \C$.
If $Q(2,1) \cap P \cong Q(2,0)$ then $f$ must be linear.  If
$Q(2,1) \cap P \cong S^1 \times \C$ then $f$ cannot be an immersion
(essentially
because the inverse image of $\{ p \} \times \C$ cannot be contained in
a sphere).  We have already handled this case.

If $f$ is of type $D$, then $f$ is not self-dual by an easy computation.
Being self-dual is a requirement for maps to take a hyperquadric
to a hyperquadric, so $f$ is not equivalent to any map of hyperquadrics.
See Faran~\cite{Faran:B2B3} for more details.

Therefore, $f$ must be rational of degree less than or equal to
$3$. If $f$ is linear then it is obviously Q-equivalent to $(z,w) \mapsto
(z,w,0)$.
If $f$ is of degree two, we apply Lemma~\ref{lemma:deg2hq23}.
If $f$ is of degree three,
we apply Lemma~\ref{lemma:deg3}.
\end{proof}


\def\MR#1{\relax\ifhmode\unskip\spacefactor3000 \space\fi%
  \href{http://www.ams.org/mathscinet-getitem?mr=#1}{MR#1}}

\end{document}